\documentclass[12pt]{amsart}
\font\bbbld=msbm10 scaled\magstephalf

\newcommand{\bfR}{\hbox{\bbbld R}}

\newcommand{\e}{\varepsilon}
\newcommand{\goto}{\rightarrow}
\newcommand{\ol}{\overline}
\newcommand{\ul}{\underline}
\newcommand{\be}{\begin{equation}}
\newcommand{\ee}{\end{equation}}
\newcommand{\bea}{\begin{eqnarray}}
\newcommand{\eea}{\end{eqnarray}}

\newcommand{\lu}{\underline{u}}
\newcommand{\tC}{\tilde{C}}

\newtheorem{theorem}{Theorem}[section]
\newtheorem{lemma}[theorem]{Lemma}
\newtheorem{proposition}[theorem]{Proposition}

\theoremstyle{definition}
\newtheorem{definition}[theorem]{Definition}

\theoremstyle{remark}

\numberwithin{equation}{section}

\begin{document}
\setlength{\baselineskip}{1.2\baselineskip}

\title[Entire translating solitons in Minkowski space]
{Entire downward translating solitons to the mean curvature flow in Minkowski space}

\author{Joel Spruck}
\thanks{Research of the first author is partially supported in part by the NSF}
\address{Department of Mathematics, Johns Hopkins University,
 Baltimore, MD 21218}
\email{js@math.jhu.edu}
\author{Ling Xiao}
\address{Department of Mathematics, Rutgers University,
Piscataway, NJ 08854}
\email{lx70@math.rutgers.edu}

\begin{abstract}
In this paper, we study entire translating solutions $u(x)$ to a mean curvature flow  equation in
Minkowski space.  We show that if $\Sigma=\{(x, u(x)) | x\in\mathbb{R}^n\}$ is a strictly spacelike hypersurface, then $\Sigma$ reduces to a strictly convex rank $k$
soliton  in $\bfR^{k,1}$ (after splitting off trivial factors) whose ``blowdown'' converges to a multiple $ \lambda \in (0,1)$  of a positively homogeneous degree one convex function in $\bfR^k$. We also show that there is nonuniqueness as the rotationally symmetric solution may be perturbed to a solution by an arbitrary smooth order one perturbation.
\end{abstract}

\maketitle

\section{Introduction}
\label{sec1}

Let $\mathbb{R}^{n,1}$ be the Minkowski space with Lorentz metric
\[\bar{g}=\sum\limits_{i=1}^ndx_i^2-dx_{n+1}^2.\]
We will say that a hypersurface $\Sigma=\{(x, u(x)) |x\in\Omega\}\subset\mathbb{R}^{n, 1}$ is strictly spacelike
if $u\in C^1(\Omega)$ and $|Du|\leq c_0<1$ in $\Omega.$

Ecker and Huisken \cite{EH} studied the mean curvature flow with forcing term in Minkowski space and proved longtime existence.  More specifically,
they studied the  equation
\be\label{eq1.01}
\frac{\partial u}{\partial t}=\sqrt{1-|Du|^2}\left[\text{div}\left(\frac{Du}{\sqrt{1-|Du|^2}}\right)-\mathcal{H}\right],
\ee
where $u = u(x, t)$ is the height function and $\mathcal{H} = H(x)$ is the forcing term.
Later, M. Aarons \cite{Aar} proved the following convergence result.
\begin{theorem}\label{thm1.1}(\cite{Aar})
Let $M_0$ be a smooth spacelike hypersurface with bounded curvature. Suppose $M_0$ never intersects  future null infinity $I^+$ or past null infinity $I^-$. Then $M_t$ converges under the flow \eqref{eq1.01} to a convex downward translating
soliton, that is, an entire solution of $H=c+a\frac1{\sqrt{1+|\nabla u|^2}},\, a<0$.
\end{theorem}

 After a rescaling, we may assume $a=-1$ and consider $H=C-\frac{1}{\sqrt{1-|Du|^2}},$ where $C>1$ is a constant.
We obtain
\be \label{eq1.05}
 \text{div}\left(\frac{\nabla u}{\sqrt{1-|Du|^2}}\right)=C-\frac1{\sqrt{1-|Du|^2}}
 \ee
or in nondivergence form
\be \label{eq1.1}
\left(\delta_{ij}+\frac{u_iu_j}{1-|Du|^2}\right)u_{ij}=C\sqrt{1-|Du|^2}-1.
\ee
Notice that any $u=\vec{v}\cdot x +b,\, |v| =\sqrt{1-\frac1{C^2}}$ (a maximal hypersurface) is a solution of \eqref{eq1.1}. The existence of a unique (up to translation) radial symmetric solution of \eqref{eq1.1} was shown by Ju, Lu and Jian \cite{JLJ}.

Aarons  \cite{Aar} in fact conjectured that any solution
$u$ of \eqref{eq1.05} is either rotationally symmetric about some point $x_0$ or
is a hyperplane. However, this conjecture is not correct.
Let $x'=(x_1, \ldots, x_k) $ and set
$u(x) =\sum_{i=k+1}^n a_i x_i+h(x')$
where $h$ is strictly convex in $x'$. Then $u$ satisfies \eqref{eq1.1} if and only if
 $h$ satisfies
\be \label{eq1.06} \sum_{i,j=1}^k (\delta_{ij}+\frac{h_i h_j}{1-|a|^2-|Dh|^2})h_{ij}
=C\sqrt{1-|a|^2} \sqrt{1-\frac{|Dh|^2}{1-|a|^2} }\ -1~.\ee

Now let $ \tilde h =\frac1{\lambda^2}h(\lambda xÕ),\,\, \tilde{C}=\lambda C>1$
with $\lambda=\sqrt{1-|a|^2}$.
Then $\tilde h$ satisfies
\be \label{eq1.07} \sum_{i, j=1}^k (\delta_{ij}+\frac{ {\tilde h}_i {\tilde h}_j}{1-|D{\tilde h}|^2}){\tilde h}_{ij}
=\tilde{C} \sqrt{1-|D{\tilde h}|^2}\ -1~.\ee

Thus $\tilde h$ is a rank $k$ solution of \eqref{eq1.1} in $\bfR^{k}$ with $C$ replaced by $\tilde{C}=\lambda C>1$.

In fact, we will show a splitting theorem analogous to what Choi-Treibergs \cite{CT} proved for spacelike  constant mean curvature hypersurfaces.
\begin{theorem} \label{th1.0}
 Let  $u$ be a strictly spacelike solution of \eqref{eq1.05} and let $\Sigma=\{(x, u(x)) | x\in\mathbb{R}^n\}$ be the graph of $u$. Then $\Sigma$ is convex with uniformly bounded second fundamental form. Moreover after an $R^{n,1}$ rigid motion, $R^{n,1}$ splits as a product $R^{k,1} \times R^{n-k}$
such that $\Sigma$ also splits as a product $\Sigma^{k}\times R^{n-k} $ where $\Sigma^{k}=R^{n,1}\cup R^{k,1}$ is a strictly convex
graphical solution in $\bfR^{k,1}$
 \end{theorem}

  Thus, it is  natural to ask if Aarons conjecture correct for $u,$  a strictly convex solution of \eqref{eq1.05} in $\mathbb{R}^n$ ?
In other words, is $\Sigma=\{(x, u(x))| x\in \mathbb{R}^n\}$ rotationally symmetric? The answer is no.

\begin{theorem}\label{th2.0}
Let $f\in C^2(S^{n-1}_{\tC}),\, \tC=\sqrt{1-(\frac1C)^2}$. Then there exists an entire strictly spacelike hypersurface $u$
satisfying equation \eqref{eq1.05} such that
\[u(x)\goto \tC |x|-\frac{n-1}{C^2}\log |x|+f(\tC x)\,\,\mbox{as $|x|\goto\infty$}.\]
\end{theorem}

As in the work of Treibergs \cite{T} and Choi-Treibergs \cite{CT}, the blow-down
 of a convex strictly spacelike solution $V_u=\lim_{r\goto \infty}\frac{u(rx)}r$ converges uniformly on compact subsets to
the space $\tC \mathcal{Q}$ of convex homogeneous degree one convex functions
whose gradient has magnitude $\tC$ wherever defined. It was shown in \cite{CT}
that the space $\mathcal{Q}$ is in one to one correspondence with the  set of {\em lightlike} directions
\[L_u:=\{x\in S^{n-1}: V_u(x)=1\}~.\]
It may be possible that  any cone in $\tC \mathcal{Q}$ arises as the blow-down of a solution to \eqref{eq1.05} but we have not  shown this.

An outline of the paper is as follows. In section \ref{sec2} we show the strictly space like assumption implies that the graph $\Sigma$ is mean convex. Then in section \ref{sec3} we show $\Sigma$ is in fact convex and then prove the splitting Theorem \ref{th1.0}.  In section \ref{sec4} we study the blow-down $V_u$ and finally in section \ref{sec5}
following \cite{T}, we construct counterexamples for the radial cone in $\tC \mathcal{Q}$ and prove Theorem \ref{th2.0}

\section{Strictly spacelike implies mean convex}
\label{sec2}
\setcounter{equation}{0}
Let $a^{ij}=\delta_{ij}+\frac{u_iu_j}{w^2},$ where $w=(1-|Du|^2)^{1/2}; $
then equation \eqref{eq1.1} becomes
\be\label{eq2.1}
a^{ij}u_{ij}=Cw-1
\ee
Then $w_i=-\frac{u_ku_{ki}}{w}$ and
\be\label{eq2.2}
\begin{aligned}
w_{ij}&=-\frac{u_{ki}u_{kj}}{w}-\frac{u_ku_{kij}}{w}+\frac{u_ku_{ki}w_j}{w^2}\\
&=-\frac{u_ku_{kij}}{w}-\frac{1}{w}\left(u_{ki}u_{kj}+\frac{u_ku_{ki}u_lu_{lj}}{w^2}\right)\\
&=-\frac{u_ku_{kij}}{w}-\frac{1}{w}a^{kl}u_{ki}u_{lj}.\\
\end{aligned}
\ee
\begin{lemma}\label{lem2.1}
Suppose $\Sigma=\{(x, u(x)) | x\in\mathbb{R}^n\}$ is a strictly spacelike hypersurface, and $u(x)$
satisfies equation \eqref{eq1.1}. Then $\Sigma$ is mean convex, that is $H\geq 0.$
\end{lemma}
\begin{proof}
We differentiate equation \eqref{eq2.1} with respect to $x_k$ to obtain
\be\label{eq2.3}
(a^{ij})_ku_{ij}+a^{ij}u_{ijk}=Cw_k.
\ee
Since
\be\label{eq2.4}
\begin{aligned}
(a^{ij})_ku_{ij}&=\left(\frac{u_{ik}u_j}{w^2}+\frac{u_iu_{jk}}{w^2}-2\frac{u_iu_jw_k}{w^3}\right)u_{ij}\\
&=\frac{2}{w}\left(\frac{u_{ik}u_j}{w}+\frac{u_iu_ju_lu_{lk}}{w^3}\right)u_{ij}\\
&=\frac{2}{w}\left(-w_iu_{ik}-\frac{u_iu_lu_{lk}}{w^2}w_i\right)\\
&=-\frac{2}{w}\left(\delta_{ij}+\frac{u_iu_j}{w^2}\right)w_iu_{kj}\\
&=-\frac{2}{w}a^{ij}w_iu_{kj},\\
\end{aligned}
\ee
and $u_{ijk}=u_{kij}$, this gives
\be\label{eq2.5}
a^{ij}u_{kij}-\frac{2}{w}a^{ij}w_iu_{kj}=Cw_k.
\ee
Multiplying \eqref{eq2.5} by $\frac{u_k}w$
and using
$\frac{u_ku_{kij}}{w}=-w_{ij}-\frac{1}{w}a^{kl}u_{ki}u_{lj},$ we obtain
\be\label{eq2.55}
a^{ij}w_{ij}-2\frac{a^{ij}w_iw_j}w+C \frac{u_k}w w_k=-\frac{1}{w}a^{ij}a^{kl}u_{ki}u_{lj}.
\ee

We now observe that since  $|A|^2=\frac{1}{w^2}a^{ij}a^{kl}u_{ki}u_{lj}$ we can rewrite
\eqref{eq2.55} as
\be \label{eq2.72}
a^{ij}(\frac1w)_{ij}+C\frac{u_k}w (\frac1w)_k=|A|^2 \frac1w~.
\ee

The Omori-Yau maximum principle (see for example \cite{Y75}, \cite{PGS}) implies that $\frac1w$ achieves its maximum at infinity and moreover, there
exists a sequence $\{P_N\}$ such that $\frac1w(P_N)\goto \sup{\frac1w},\, |\nabla (\frac1w)|(P_N)<1/N,$ and $(\frac1w)_{ij}(P_N)\geq-1/N \delta_{ij}$.
Therefore,
\be\label{eq2.8}
\frac1{n} H^2  \frac1w \leq |A|^2 \frac1{w} \leq \frac{C_1}N\frac{1}{w} \hspace{.2in} \text{at $P_N$}.
\ee
Thus
 $H(P_N)\rightarrow 0$ at infinity. Since  $H=C-\frac1w$  we obtain $\inf{H}=0$ .
\end{proof}

%

\section{Mean convexity implies convexity and constant rank}
\label{sec3}
\setcounter{equation}{0}
In this section, we will use ideas due to Hamilton \cite{Ham} to prove that under the strictly spacelike assumption,
$\Sigma$ is in fact convex. We use the following approximation of Heidusch\cite{Heid}.

\begin{definition}
\label{def3.1}
The $\delta$-approximation to the function $\min(x_1, x_2)$ is given by
\[\mu_2(x_1,x_2)=\frac{x_1+x_2}{2}-\sqrt{\left(\frac{x_1-x_2}{2}\right)^2+\delta^2}\]
for any $\delta>0.$ The $\delta$-approximation to the function $\min(x_1,x_2,\cdots,x_n)$
is defined recursively by
\[\mu_n=\frac{1}{n}\sum_{i=1}^n\mu_2(x_i,\,\mu_{n-1}(x_1,\cdots,x_{i-1},x_{i+1},\cdots,x_n))\]
\end{definition}

The following lemma \ is elementary (see \cite{Aar}).
\begin{lemma}
\label{lem3.1}
For every $\delta>0$ and $n\geq 2$ we have:\\
1. $\mu_n$ is smooth, symmetric, monotonically increasing and concave.\\
2. $\frac{\partial \mu_n}{\partial x_i}\leq 1$.\\
3. $\min(x_1,\cdots, x_n)-n\delta\leq \mu_n\leq\min(x_1, \cdots, x_n)$.\\
4. For $x\in\mathbb{R}^n$ we have
\[
\mu_n\leq\sum_{i=1}^n\frac{\partial \mu_n}{\partial x_i}x_i\leq \mu_n+n\delta,\]
\,\,\,and $\sum_{i=1}^n\frac{\partial \mu_n}{\partial x_i}x_i^2\geq \mu_n^2-n\delta^2-\frac{n\delta}{4}\sum\limits_{1\leq i<j\leq n}|x_i+x_j|.$
\end{lemma}

\begin{lemma}
\label{lem3.2} Assume $\Sigma=\{(x, u(x))|x\in\mathbb{R}^n\}$ is mean convex, and u satisfies equation
\eqref{eq1.1}. Then the principal curvatures of  $A$ are nonnegative, i.e. $\Sigma$ is convex.
\end{lemma}
\begin{proof}
Let p be a fixed point in $\Sigma$ (we may assume $p=(0,0)$) and let $r$ be the distance function from p restricted to the geodesic ball $B^{\Sigma}(p,a)$ of radius $a$ centered at p (in the induced metric on $\Sigma$).
Let $f(x)=|A|^2=\sum_{i,j}h_{ij}^2$.  By a well-known calculation (see equation (2.24) of \cite{CY76})
\be \label{eq3.5}
\frac{1}{2}\triangle\left(\sum_{i,j}h_{ij}^2=\right)
=\sum_{i,j,k}h_{ijk}^2+\sum_{i,j}h_{ij}H_{ij}+\left(\sum_{i,j}h_{ij}^2\right)^2
-\left(\sum_{i,j,k}h_{ij}h_{jm}h_{mi}\right)H
\ee
\be\label{eq3.6}
\begin{aligned}
H_{ij}&=\nabla_i\nabla_jH=\nabla_i\left(\nabla_j\left<\nu, e_{n+1}\right>\right)\\
&=\nabla_i\left<h_{jk}\tau_k, e_{n+1}\right>=-h_{ijk}u_k-h_{ik}h_{jk}\nu^{n+1}.
\end{aligned}
\ee
In the following, we will denote $\nu^{n+1}$ by $V.$
\be\label{eq3.7}
\begin{aligned}
\frac{1}{2}\triangle f&=\sum_{i,j,k}h^2_{ijk}-h_{ij}h_{ijk}u_k
-h_{ij}h_{ik}h_{jk}V+f^2-\left(\sum_{i,j,m}h_{ij}h_{jm}h_{mi}\right)(C-V)\\
&\geq-\frac{f(V^2-1)}{4}+f^2-Cf^{3/2}\\
&\geq-\frac{C^2}{4}f+f^2-Cf^{3/2}\geq \frac12 f^2- C_1f\\
\end{aligned}
\ee
Let $\eta(x)=a^2-r^2$ and set $g=\eta^2 f$. Then in $B^{\Sigma}(p,a)$,
\be \label{eq3.8} \frac12 (\eta^{-2}g )^2  \leq C_1\eta^{-2}g+ \Delta (\eta^{-2}g)=
C_1\eta^{-2}g+\eta^{-2} \Delta g -2\eta^{-3}<\nabla \eta, \nabla g>+ g\Delta (\eta^{-2})~.
\ee
At the point $\ol{x}$ where $g$ assumes its maximum, $\nabla g=0$ and $\Delta g \leq 0$.
Since $R_{ii}\geq -\frac{H^2}4\geq -\frac{C^2}4$, we have by Lemma 1 of \cite{Y75} that $\Delta r^2 \leq C_3 (1+r^2)$.
Hence at  $\ol{x}$,
\be \label{eq3.9} \begin{aligned}
&\frac12 g^2 \leq C_1 \eta^2g +g \eta^4 \Delta (\eta^{-2})=
C_1 \eta^2g-2g\eta\Delta \eta +6g |\nabla \eta|^2 \\
 &\leq C_2g(a^4+2a^2|\Delta r^2|+24r^2)
\end{aligned} \ee
It follows that $g(\ol{x})\leq C_5a^4.$ Therefore, let $a\goto\infty$ we get, $|A|^2\leq C_6.$

Next we will show that the smallest principal curvature $\lambda_{\text{min}}$ of $\Sigma$ is nonnegative.
Let $\mu_n(\lambda_1,\cdots, \lambda_n)=F(\gamma^{ik}h_{kl}\gamma^{lj}),$
assume $\mu_n$ achieves its minimum at an interior point $x_0.$
Then at this point we have
\be\label{eq3.12}
\begin{aligned}
\triangle \mu_n&=F^{ij}h_{ijkk}+F^{rl,st}h_{rlk}h_{stk}\\
&=F^{ij}(H_{ij}-Hh_{ij}^2+h_{ij}h_{kk}^2)+F^{rl,st}h_{rlk}h_{stk}\\
&\leq F^{ij}\nabla_kh^j_i\left<\tau_k, e_{n+1}\right>-F^{ij}h_i^kh_{jk}\nu^{n+1}-HF^{ij}h_{ij}^2
+(\mu_n+n\delta)|A|^2\\
&\leq \left<\nabla_k\mu_n, e_{n+1}\right>-\mu_n^2+n\delta^2+
\frac{n\delta}{4}\sum\limits_{1\leq i<j\leq n}|\lambda_i+\lambda_j|\\
&+H\left(-\mu_n^2+n\delta^2+\frac{n\delta}{4}\sum\limits_{1\leq i<j\leq n}|\lambda_i+\lambda_j|\right)
+(\mu_n+n\delta)|A|^2.\\
\end{aligned}
\ee
Thus,
\be\label{eq3.13}
0\leq-\mu_n^2+(\mu_n+n\delta)|A|^2+n\delta^2+\frac{n\delta}{4}\sum\limits_{1\leq i<j\leq n}|\lambda_i+\lambda_j|.
\ee
Letting  $\delta\rightarrow 0$ we find,
\be\label{eq3.14}
\lambda_{\text{min}}^2\leq \lambda_{\text{min}}|A|^2,
\ee
which implies that $ \lambda_{\text{min}}\geq 0$.

Since we have already proven that $|A|^2$ is bounded, we can again apply the Omori-Yau maximum principle (this time on $\Sigma$) and show that,
if $\mu_n$ achieves its minimum at infinity then $\mu_n\geq 0.$
This completes the  proof that mean convexity implies convexity.
\end{proof}

Now that we have proved convexity, we  prove the splitting Theorem \ref{th1.0} of the introduction. \\
{\bf Proof of Theorem \ref{th1.0}.}  Suppose that for some unit vector $\vec{v}$ and some $x_0\in \bfR^n, \, D_{\vec{v}}^2 u(x_0)=0$. Applying an isometry (boost transformation) of $\bfR^{n,1}$, may assume $x_0=0,\, \vec{v}=e_n,\, Du(0)=0,\,u_{nn}(0)=0$ and $u_{ij}$ is nonnegative.
Rewrite \eqref{eq1.1} as
\be \label{eq3.15}
\Delta u=-\frac{u_iu_j}{1-|Du|^2}u_{ij}+C\sqrt{1-|Du|^2}-1
\ee
Differentiating \eqref{eq3.15} twice in the $x_n$ direction, we  can apply the argument
of Korevaar (a special case of \cite{KL}) exactly as in Theorem 3.1 of Choi-Treibergs \cite{CT} to conclude $u_{nn}\equiv 0$ and $\Sigma$ is ruled by lines parallel to the $x_n$ axis. Therefore $\Sigma=\Sigma^{n-1}\times \bfR^1$ and also $\bfR^n=\bfR^{n-1}\times \bfR^1$. Therefore $u=ax_n+h(x'),\, x'=(x_1, \ldots, x_{n-1})$ where
\be \label{eq3.20} \sum_{i,j=1}^{n-1} (\delta_{ij}+\frac{h_i h_j}{1-|a|^2-|Dh|^2})h_{ij}
=C\sqrt{1-|a|^2} \sqrt{1-\frac{|Dh|^2}{1-|a|^2} }\ -1~.\ee

Now let $ \tilde h =\frac1{\lambda^2}h(\lambda xÕ),\,\, \tilde{C}=\lambda C>1$
with $\lambda=\sqrt{1-|a|^2}$.
Then $\tilde h$ satisfies
\be \label{eq13.25} \sum_{i, j=1}^{n-1} (\delta_{ij}+\frac{ {\tilde h}_i {\tilde h}_j}{1-|D{\tilde h}|^2}){\tilde h}_{ij}
=\tilde{C} \sqrt{1-|D{\tilde h}|^2}\ -1~.\ee
 Proceeding inductively  completes the proof of Theorem \ref{th1.0}.

\section{The asymptotic cone at infinity.}
\label{sec4}
\setcounter{equation}{0}

In this section, we will study the asymptotic behavior of  $u$ at infinity.

\begin{proposition}Let u be a convex space like solution of \eqref{eq1.1}.
Assume $u(0)=0$ and denote $u^h(x)=\frac{u(hx)}{h}$. Define $V_u(x)=\lim\limits
 _{h\goto \infty}{u^h(x)}$ then $V_u(x)$ exists for all $x$ and is a positively homogeneous degree one convex function. Moreover
for all $x\in \bfR^n$ and  $\delta>0$ there exists $y\in \bfR^n$ so that $|y-x|=\delta$ and $|V_u(x)-V_u(y)|=\sqrt{1-\frac{1}{C^2}}\hspace{.05in} \delta$. In particular $|DV_u(x)|=\sqrt{1-\frac{1}{C^2}}$ at every point of differentiability of $V_u(x)$.
 \end{proposition}

\begin{proof}
Note that since $u$ is convex,
 $0=u(0)  \geq u(hx)-\sum_{i=1}^n hx_i u_{x_i}(hx)$ we get $\frac{d}{dh}u^h(x) \geq 0$.
  Then $V_u(x)$,  the projective boundary values (blow-down)
 of $u$ at infinity in the terminology of Treibergs \cite{T} and Choi-Treibergs \cite{CT}, is well-defined, strictly spacelike, convex on $\bfR^n$ and satisfies
 \[ V_u(\lambda x) =\lambda V_u(x), \hspace{.1in} \lambda>0, \]
 \[ |V_u(x)-V_u(y)| \leq \sqrt{1-\frac{1}{C^2}}\hspace{.05in} |x-y| ~.\]
  Claim: for all $x\in \bfR^n$ and  $\delta>0$ there exists $y\in \bfR^n$ so that $|y-x|=\delta$ and $|V_u(x)-V_u(y)|=\sqrt{1-\frac{1}{C^2}}\hspace{.05in} \delta$.
 Suppose the claim is false. Then  there exists $x\in \bfR^n$ and $\e>0$ such that
 \[ V_u(y)\leq V_u(x)+(1-2\e)\sqrt{1-\frac{1}{C^2}}\hspace{.05in} \delta \hspace{.1in} \forall y\in \partial B(x,\delta)~.\]
 Since $u^h(x)\goto V_u(x)$ uniformly on compact subsets, we may choose $h_0$ large so that for all $h>h_0$,
 \[ u^h(y)\leq V_u(x)+(1-\e)\sqrt{1-\frac{1}{C^2}}\hspace{.05in} \delta \hspace{.1in} \forall y\in \partial B(x,\delta)~.\]
Now $u^h(y)$ satisfies
\be \label{eq4.15}
H^h:=\text{div} \left(\frac{Du^h}{\sqrt{1-|Du^h|^2}}\right)
=h (C-\frac1{{\sqrt{1-|Du^h|^2}}})\geq 0 \hspace{.1in} \text{in $B(x, \delta)$}
\ee
We now make use of the radial solutions of the maximal surface equation $H=0$ introduced by Bartnik and Simon
\cite{BS}. Consider the barrier
\[w(y)= V_u(x)+(1-\e)\sqrt{1-\frac{1}{C^2}}\hspace{.05in} \delta +\int_0^{|y-x|} \frac{h}{\sqrt{t^{2n-2}+h^2}}~dt
- \sqrt{1-\frac{1}{C^2}} \int_0^{\delta}\frac{h}{\sqrt{t^{2n-2}+h^2}}~dt~.\]
Note that on $\partial B(x,\delta)$,
\[w(y)-u^h(y)\geq (1-\sqrt{1-\frac{1}{C^2}}) \int_0^{\delta}\frac{h}{\sqrt{t^{2n-2}+h^2}}~dt>0~.\]
Hence by the maximum principle, $u^h(y)<w(y) \hspace{.1in} \text{in $B(x,\delta)$}$.
In particular at $y=x$,
\[ u^h(x)< V_u(x)+(1-\e)\sqrt{1-\frac{1}{C^2}}\hspace{.05in} \delta -\sqrt{1-\frac{1}{C^2}} \int_0^{\delta}\frac{h}{\sqrt{t^{2n-2}+h^2}}~dt~.\]
Now let $h\goto \infty$ to conclude $V_u(x) \leq V_u(x)-\e \sqrt{1-\frac{1}{C^2}}\hspace{.05in} \delta$, a contradiction, so the claim is proven and the proposition
is complete.
\end{proof}

\section{Construction of Counterexamples.}
\label{sec5}
\setcounter{equation}{0}

We will follow Treiberg's idea (see \cite{T}) to construct counterexamples, more prŽcisely we will construct solution to
equation \eqref{eq1.05} such that $u(x)\rightarrow \tC|x|-\frac{n-1}{C^2}\log|x|+f\left(\frac{\tC x}{|x|}\right), $
as $|x|\goto\infty,$ where $\tC=\sqrt{1-\frac{1}{C^2}}$ and $f\in C^2(S^{n-1}_{\tC}).$

We extend the function f to $\mathbb{R}^n\setminus\{0\}$ by defining
$f(\tC x)=f\left(\frac{\tC x}{|x|}\right).$ Since $f\in C^2$, we have
 for all $x, y\in S^{n-1}:$ \be\label{eq5.1}
|f(\tC x)-f(\tC y)-Df(\tC y)(\tC x-\tC y)|\leq M|\tC x-\tC y|^2=-2\tC My\cdot(\tC x-\tC y).
\ee
Let $p_1(\tC y)=Df(\tC y)+2M\tC y$ and $p_2(\tC y)=Df(\tC y)-2M\tC y,$ so that
\be\label{eq5.2}
p_1(\tC y)\cdot(\tC x-\tC y)\leq f(\tC x)-f(\tC y)\leq p_2(\tC y)\cdot(\tC x-\tC y).
\ee
Now let $\psi(x)$ denote the rotationally symmetric solution to \eqref{eq1.05} (see \cite {JLJ}). We know that
$\psi(x)\goto\tC |x|-\frac{n-1}{C^2}\log|x|+o(1)$ as $|x|\goto\infty.$
Let $z_1(x;\tC y)=f(\tC y)-p_1(\tC y)\cdot \tC y+\psi(x+p_1(\tC y))$
and $z_2(x;\tC y)=f(\tC y)-p_2(\tC y)\cdot \tC y+\psi(x+p_2(\tC y)).$ Then by equation
\eqref{eq5.2} we have
\be\label{eq5.3}
f(\tC x)\geq z_1(rx;\tC y)-\tC r+\frac{n-1}{C^2}\log r\,\,\mbox{as $r\goto\infty,\, x, y\in S^{n-1}$,}
\ee
and
\be\label{eq5.4}
f(\tC x)\leq z_2(rx;\tC y)-\tC r+\frac{n-1}{C^2}\log r\,\,\mbox{as $r\goto\infty,$ $x, y\in S^{n-1}$.}
\ee
Therefore,
\be\label{eq5.5}
\begin{aligned}
&\lim\limits_{r\goto\infty}z_1(rx;\tC y)-\tC r+\frac{n-1}{C^2}\log r\leq f(\tC x)\\
&\leq\lim\limits_{r\goto\infty}z_2(rx;\tC y)-\tC r+\frac{n-1}{C^2}\log r
\end{aligned}
\ee
for $x\in S^{n-1}.$

Let $q_1(x)=\sup\limits_{y\in S^{n-1}} z_1(x;\tC y)$ and $q_2(x)=\inf\limits_{y\in S^{n-1}}z_2(x;\tC y).$
Then, $q_1(x)\leq q_2(x)$ and $q_i(x)$ (i=1,2) tends to $f(\tC x)+\tC r-\frac{n-1}{C^2}\log r$
as $r\goto\infty.$

\begin{lemma}\label{lem5.1}
There exists a smooth solution $u$ to the Dirichlet problem
\begin{equation}\label{eq5.6}
\left\{
\begin{aligned}
&a^{ij}u_{ij}-Cw+1=0\,\, \mbox{in G}\\
&u=0 \,\,\mbox{on $\partial G$}
\end{aligned}
\right.
\end{equation}
where $G$ is a convex $C^{2,\alpha}$ domain in $\mathbb{R}^n.$
\end{lemma}
\begin{proof}
Let $d= \text{diam} (G)$ be the diameter of $G.$ For any $y\in \partial G,$ we can choose  coordinates such that
$y=(y_1, 0,\cdots, 0)$ and $G \subset\{x\mid |x_1|\leq y_1\}$, where $0<y_1\leq d/2$.
Let $\underline{u}(x)=\tC x_1-\tC y_1,\, \bar{u}\equiv 0$. Then  $\ul{u} \leq \bar{u}$  in
$G$ and $\ul{u} $ satisfies
\be\label{eq5.7}
a^{ij}\lu_{ij}-Cw_{\lu}+1=0
\ee
By the maximum principle, any solution $u$ to the Dirichlet problem \eqref{eq5.6} satisfies
\be\label{eq5.8}
\lu \leq u\leq \bar{u}\, \mbox{on $G$}.
\ee
so $|Du(x)|\leq\tC\,\,\mbox{on $\partial \bar{G}$}$.  Combined with equation \eqref{eq2.72}  we conclude that
\be\label{eq5.9}
|Du(x)|\leq\tC\,\,\mbox{on $\bar{G}$.}
\ee
Now it is standard (see \cite{GT}) to prove that a smooth solution $u\in C^{2,\alpha}(\ol{G})$ exists.
\end{proof}

Finally, we will find a sanwiched solution u such that
\[q_1\leq u<q_2.\]

  Let $\phi$ be a strictly spacelike hypersurface $q_1\leq\phi<q_2$ so that
$\phi(0)=q_1(0)$ and $G_m=\phi^{-1}((-\infty, m))$ is a convex domain with $C^{2,\alpha}$ boundary.
By lemma \ref{lem5.1} we know there is an analytic solution $u_m$ to the Dirichlet problem
\be\label{eq5.10}
\begin{aligned}
&a^{ij}u_{ij}-C w+1=0\,\,\mbox{on $G_m$}\\
&u=m\,\,\mbox{on $\partial G_m$}.\\
\end{aligned}
\ee
Therefore, we find a sequence of finite solutions $u_m$ with
$q_1\leq u_m<q_2$ defined on convex domains $G_m$ which exhaust $\mathbb{R}^n.$

  Next, let K be a compact subset of $\mathbb{R}^n.$ Then, by equation $\eqref{eq5.9}$ there are constants $r_1<r_2$
so that for sufficiently large $m$ we have
\[dist_m(0, x)<r_1,\,\mbox{for all $x\in K$},\]
\[dist_m(0, x)<r_2,\,\mbox{for all $x\in\partial G_m$},\]
where $dist_m(0,x)$ is the intrinsic distance between the points $(0, u_m(0))$
and $(x, u_m(x))$ on $\Sigma_m=\{(x, u_m(x)|x\in G_m)\}.$

  At last, following the proof of Lemma \ref{lem3.2}, we find $u_m$ has uniform $C^3$ bounds on compact subsets.
Hence, a subsequence can be extracted that converges to a global solution of equation \eqref{eq1.05}.
Moreover, $\lim\limits_{m_j\goto\infty} u_{m_j}=u$ satisfies $u(x)\goto \tC |x|-\frac{n-1}{C^2}\log |x|+f(\tC x).$
Thus we have proved

\begin{theorem}\label{thm5.1}
Let $f\in C^2(S^{n-1}_{\tC}).$ Then there exists an entire strictly spacelike hypersurface $u$
satisfying equation \eqref{eq1.05} such that
\[u(x)\goto \tC |x|-\frac{n-1}{C^2}\log |x|+f(\tC x)\,\,\mbox{as $|x|\goto\infty$}.\]
\end{theorem}

\section*{Acknowledgement}
The second author would like to thank Cornell University
for their hospitality during her visit in the spring 2014 while this project was
initiated. In particular, she would like to thank Professor Xiaodong Cao and Professor Laurent Saloff-Coste for
their help and support.

\end{document}